\documentclass[12pt]{amsart}
\usepackage[margin=1in]{geometry}
\usepackage{amsmath, amssymb, amsthm, bbm, mathtools,bm}
\usepackage{amscd,mathrsfs,graphicx} 

\newtheorem{thm}{Theorem}[section]
\newtheorem{defin}[thm]{Definition}
\newtheorem{lem}[thm]{Lemma}

\newtheorem{prop}[thm]{Proposition}
\newtheorem{rem}[thm]{Remark}
\newtheorem{coro}[thm]{\bf Corollary}

\usepackage{tcolorbox}
\usepackage{import}
\usepackage{pdfpages}
\usepackage{transparent}
\usepackage{xcolor}
\usepackage{lscape}
\usepackage{enumerate}

\usepackage[pdftex,breaklinks,colorlinks]{hyperref} 
\hypersetup{linkcolor=blue,citecolor=blue}

\usepackage{parskip,upgreek}

\theoremstyle{remark}

\theoremstyle{plain}

\renewcommand{\P}{{\mathbb P}}

\begin{document}
\begin{abstract}
We prove that the Poisson boundary of a random walk with finite entropy on a non-elementary hyperbolic group can be 
identified with its hyperbolic boundary, without assuming any moment condition on the measure.

We also extend our method to groups with an action by isometries on a hyperbolic metric space containing a WPD element; 
this applies to a large class of non-hyperbolic groups such as relatively hyperbolic groups, mapping class groups, and groups acting on CAT(0) spaces. 
\end{abstract}

\title{The Poisson boundary of hyperbolic groups without moment conditions}

\author{Kunal Chawla}
\address{University of Toronto, Canada}
\email{kunal.chawla@mail.utoronto.ca}

\author{Behrang Forghani}
\address{College of Charleston, USA}
\email{forghanib@cofc.edu}

\author{Joshua Frisch}
\address{ \'{E}cole Normale Sup\'{e}rieure, France}
\email{joshfrisch@gmail.com}

\author{Giulio Tiozzo}
\address{University of Toronto, Canada}
\email{tiozzo@math.toronto.edu}

\maketitle

\section{Introduction}
Let $G$ be a group that is equipped with a probability measure 
$\mu$.  The $\mu$-random walk on $G$ is the sequence of random variables 
$$w_n=g_1g_2\cdots g_n,$$ 
where $g_i$s are  independent and identically $\mu$-distributed.  One of the fundamental objects to study the long-term behavior of a random walk is the 
\emph{Poisson boundary}, that provides a representation, via the Poisson formula, of the space of bounded harmonic functions on a group. The study of the Poisson boundary goes back to the work of Blackwell \cite{Blackwell1955} and  Feller \cite{Feller};  however, Furstenberg \cite{Furstenberg-semi,Fu70}  extensively studied the Poisson boundary as a tool to prove rigidity results.  

One of the main and challenging problems in the theory of random walks has been identifying the Poisson boundary of a given random walk on a given group. 
For instance, so-called \emph{Choquet-Deny groups}, that is, groups whose Poisson boundary is trivial for all measures, have been characterized recently by \cite{Josh19}. 
At the opposite end,  when a group carries a natural compactification (for example, hyperbolic groups), it is natural to ask if the compactification with the hitting measure is a model for the Poisson boundary. There has been extensive research devoted to the Poisson boundary identification problem \cite{Ledrappier-Lie, Cartwright-Kaimanovich-Woess,Kaim-Masur, KM, Ledrappier2001, Kaimanovich-Woess,  Maher-Tiozzo, 
KE19, Lyons-Peres, BG-lcg},  and the references therein.

One of the most significant steps is due to Kaimanovich \cite{Kaimanovich-hyp94} who, under the hypothesis of finite entropy and finite logarithmic moment, proved the identification of the Poisson boundary of hyperbolic groups with their hyperbolic boundary.  
His proof uses the \emph{strip approximation} technique \cite{Ka}, and has been widely extended to many classes of groups,  
but heavily depends on the finite logarithmic moment assumption.

Going beyond these assumptions, \cite{BG-semi} identified the Poisson boundary for random walks on a free semigroup, under the condition that the projection to the distance to the origin has finite entropy. This, for instance, holds if the measure has either finite entropy or finite logarithmic moment. Note that, without any assumption on the measure, the identification problem is still open, even on a free semigroup. 
For free groups, the Poisson boundary identification is also unsolved beyond finite entropy and finite logarithmic moment. 
The proof for the free semigroup relies on the fact that the length at time $n$ is a random walk, which in turn relies on the fact that length is a semi-group homomorphism. 
This is no longer the case even for the free group, so a new method is required.

The goal of this paper is to solve the identification problem for random walks on hyperbolic groups with finite entropy, 
without the finite logarithmic moment assumption. A finitely generated hyperbolic group is called \emph{elementary} when it is finite or virtually cyclic.

\begin{thm} \label{T:main}
Let $\mu$ be a generating probability measure with finite entropy on a non-elementary hyperbolic group $G$. 
Let $\partial G$ be the hyperbolic boundary of $G$, and let $\nu$ be the hitting measure on $\partial G$ of the random walk 
driven by $\mu$. Then the Poisson boundary of $(G, \mu)$ is given by $(\partial G, \nu)$. 
\end{thm}

To our knowledge, this theorem gives the first examples of non-Choquet-Deny 
groups where the Poisson boundary is identified for all measures with finite entropy; in particular, it is new even for free groups. 
Our proof uses a version of the zero-entropy criterion for random walks in random environment from \cite{VK-RWRE}, 
but does not use strip approximation. 

Instead, we use the following ``pin down approximation" method.  A boundary point $ \xi \in \partial G $ determines, up to bounded distance, an infinite geodesic ray $\gamma$ converging to $\xi$. 
The main idea is that, once we condition the random walk to converge to $\xi$, we can reconstruct the position of the random walk at time $n$
by adding a small amount of information, which we encode in an additional sequence of partitions $(\mathcal{P}_n)$.
Let $\mathcal{A}_n$ be the partition on the path space that assigns the same equivalence class to those sample paths with the same $n$th step $w_n$. 
We say that the sequence $(\mathcal{P}_n)$ of partitions \emph{pins down} the random walk if
the conditional entropy $H_{\xi}(\mathcal{A}_n |  \mathcal{P}_n)$  grows sublinearly in $n$. In order to construct $(\mathcal{P}_n)$, we use results from \cite{pivot}, showing that a sample path of a random walk on a hyperbolic group can be thought of as a concatenation of long geodesic-like segments, attached along so-called `pivots'. 
The partition $(\mathcal{P}_n)$ is essentially obtained by recording the distance between consecutive pivotal points. 

Since the pivotal technique is also developed in groups that need not be hyperbolic, but which act by isometries on 
hyperbolic spaces, we are able to extend Theorem \ref{T:main} beyond hyperbolic groups as follows. 
Many groups that are not hyperbolic nonetheless admit an action on a (possibly non-proper) hyperbolic space. 
However, in most cases at least some elements satisfy a much weaker properness condition:
following \cite{Bestvina-Fujiwara}, an element $g$ is called \emph{WPD (weakly properly discontinuous)} for the action on $X$ if it is loxodromic 
and the set of elements that coarsely fixes two sufficiently far points along its axis is finite (see Definition \ref{D:WPD}). 

We obtain a boundary identification for all actions with at least one WPD element.

 \begin{thm}  \label{T:WPD-action-intro}
Let $ G $ be a countable group with a non-elementary action on a geodesic hyperbolic space $ X $ with at least one WPD element. Let $ \mu $ be a generating probability measure on $ G $ with finite entropy. Let $ \partial X $ be the hyperbolic boundary of $ X $, and let $ \nu $ be the hitting measure of the random walk driven by $ \mu $. Then the space $ (\partial X, \nu) $ is the Poisson boundary of $ (G, \mu) $.
\end{thm}

The theorem applies to many non-hyperbolic groups. In particular, it applies to every \emph{acylindrically hyperbolic} group, as defined by Osin \cite{Osin}. 

Important examples are the following (for a survey, see the Appendix in \cite{Osin}): 

\begin{coro}
Let $G$ be: 
\begin{enumerate}

\item a free group $G = F_\infty$ on countably many generators, and $X$ its Cayley graph with respect to a free generating set; 

\item a countable group acting properly discontinuously on a proper hyperbolic space $X$ (for instance, the fundamental 
group of a hyperbolic manifold); 

\item a non-elementary relatively hyperbolic group, and $X$ the coned-off space; 

\item an irreducible, non-cyclic right-angled Artin group, and $X$ the contact graph;  

\item a countable group acting properly and non-elementarily by isometries on a CAT(0) space with a rank-one element, and $X$ the curtain model; 

\item a mapping class group of a surface $S$ of finite type, and $X$ the curve complex of $S$;  

\item the group $\textup{Out}(F_n)$, and $X$ the free factor complex; 

\item a countable subgroup of the Cremona group containing a WPD element, and $X$ the Picard-Manin hyperboloid.
\end{enumerate}
Let $\mu$ be a non-elementary, generating measure on $G$ with finite entropy. 
Then the hyperbolic boundary of $X$ 
equipped with the hitting measure is a model for the Poisson boundary of $(G, \mu)$. 
\end{coro}

\begin{proof}
All these results follow immediately from Theorem \ref{T:WPD-action-intro}, once we show that the group $G$ contains a WPD element.
(1) By \cite[Corollary 4.3]{Minasyan-Osin}, every non-trivial element of $F_\infty$ acts as a WPD element on its Cayley graph.
For (2), note that a properly discontinuous action on a proper metric space is always WPD. 
(3) By \cite[Proposition 5.2]{Osin}, the action of a non-virtually cyclic relatively hyperbolic group on its coned-off space is acylindrical, hence WPD.
(4) For the action of right-angled Artin groups on the contact graph, acylindricality is due to \cite{BHS}, following \cite{Kim-Koberda}. 
(5) As shown by \cite{Sisto-contracting}, a group acting properly on a CAT(0) space with a rank-1 isometry is acylindrically hyperbolic.
Recently, \cite{PSZ} construct for any group of isometries of a CAT(0) space a hyperbolic space, called the \emph{curtain model}, where all rank-1 isometries 
act as loxodromic WPD elements. 
(6) The action of the mapping class group on the curve complex is acylindrical by \cite{Bowditch}.
(7) The action of $Out(F_n)$ on the free factor complex contains WPD elements by \cite{Bestvina-Feighn}.
(8) WPD elements in the Cremona group have been discussed by Cantat-Lamy \cite{Cantat-Lamy}.
\end{proof}

Theorem \ref{T:WPD-action-intro} directly generalizes \cite[Theorem 1.3]{MT-WPD}, where the same result is proven under the assumption 
of finite entropy and finite logarithmic moment, using the strip approximation.

A theory of pivots beyond actions on hyperbolic spaces, namely for any action with two independent contracting 
elements, has been very recently developed by Choi \cite{Choi}; it seems plausible that these techniques can be combined with ours to 
extend the Poisson boundary identification beyond actions on hyperbolic spaces, but we will not attempt to do this here.

\subsection{Sketch of the argument}
 
Fix a boundary point $\xi \in \partial G$, and consider the conditional random walk, conditioned to hit $\xi$ at infinity. 

Roughly, an element $(w_n)$ of the sample path is a \emph{pivot} if it lies close to the limit geodesic $[o, \xi)$, 
and moreover its increment is aligned in the direction of the geodesic ray. 

We divide the interval $[0, n]$ in time intervals $I_{k, \alpha}$ of size a large number $\alpha$. 
For each path, we keep track of some additional information, essentially as follows: 
\begin{enumerate}
\item
for each $k$, if there is a pivot in the $k$-th interval, we record the distance between the two consecutive pivots; 
\item
if there is no pivot in $I_{k, \alpha}$, then we record \emph{all} increments in the time interval $I_{k, \alpha}$, 
as well as in the intervals immediately before and after: $I_{k-1, \alpha}$ and $I_{k+1, \alpha}$.
\end{enumerate}

For each $n$ and $\epsilon > 0$, this information is encoded in a new partition, which we denote as $\mathcal{P}_n$. Then we show the following two facts: 

\begin{enumerate}

\item for each $n$, the boundary point and the additional pivotal data $\mathcal{P}_n$ is enough to pin down the location of the $n$th step $w_n$ of the walk; 
namely, the conditional entropy satisfies 
$$H_{\xi}(\mathcal{A}_n |  \mathcal{P}_n) \leq \epsilon n;$$

\item the added information $\mathcal{P}_n$ has low 
entropy, that is $H(\mathcal{P}_n) \leq \epsilon n $; 
this is in part due to the fact that, by \cite{pivot}, intervals without pivots (that we call \emph{bad intervals}) appear quite rarely.

\end{enumerate}
This implies that the conditional entropy of the walk, conditioned on hitting $\xi$, is zero, thus showing the maximality of the hyperbolic boundary.

\subsection{Structure of the paper}
Section~\ref{sec:bound} includes definitions of the Poisson boundary and conditional entropies, and a version of 
the entropy criterion. In Section~\ref{sec:hyper} we provide definitions of hyperbolic spaces and the theory of pivots. In Section~\ref{sec:pin-down} we construct a ``pin down" partition and we prove Theorem~\ref{T:main}. In the last section we show how to extend our main theorem to groups with an action by isometries on a hyperbolic metric space with a WPD element.

\subsection*{Acknowledgements}
We thank the American Institute of Mathematics for hosting us during the workshop ``Random walks beyond hyperbolic groups" in April 2022, where discussion about this problem started. We also thank Omer Angel for his contribution to the discussion. Moreover, we thank Inhyeok Choi, Vadim Kaimanovich, Samuel Taylor and Abdul Zalloum for useful comments on the first version of the paper.
J.F. is partially supported by NSF grant 2102838. G.T. is partially supported by NSERC grant RGPIN-2017-06521 and an Ontario Early Researcher Award.

\section{Boundaries and Entropy}\label{sec:bound}

\subsection{Boundaries}
We always assume $G$ is countable group.
Suppose $\mu$ is a generating probability measure on $G$, i.e. the semigroup generated by the support of $\mu$ is $G$. Let $m$ be an (auxiliary) probability measure on $G$ such that $m(g)>0$ for every element $g$ in $G$. Consider the  $\mu$-random walk on $G$ with initial distribution $m$. Let
$$
w_n = g_0g_1\cdots g_n,
$$
where  $g_i$s are independent and $g_0$ has the law of $m$ and $g_1,\cdots, g_n$ have the law of $\mu$. 
 Let $(\Omega,\P_m)$ be the space of sample paths of the $\mu$-random walk with initial distribution $m$. 
 Two sample paths $\bm w=(w_n)$ and $\bm w'=(w'_n)$ are equivalent if there exist $k$ and $k'$ such that $w_{i+k}=w'_{i+k'}$ for all $i \geq 0$. Let $\mathcal{I}$ be the $\sigma$-algebra of all measurable unions of these equivalence classes (mod $0$) with respect to $\P_m$. Rokhlin's theory of Lebesgue spaces \cite{Rokhlin52} implies that there exist a unique measurable space (up to isomorphism)  $\partial_\mu G$ with a $\sigma$-algebra $\mathcal{S}$ and a measurable map $\bm{bnd}: \Omega \to \partial_\mu G$ such that the pre-image of $\mathcal{S}$ under $\bm{bnd}$ is $\mathcal{I}$. 
 
Moreover, let us denote as $\P$ the measure on the space $\Omega$ of sample paths for the $\mu$-random walk with initial distribution the 
$\delta$-measure at the identity. 
 
 \begin{defin}
 The Poisson boundary of the random walk $(G, \mu)$ is the probability space $(\partial_\mu G, \nu)$, where $\nu$ is the image of the probability measure $\P$ under the measurable map $\bm{bnd}$.  
 \end{defin}

Note that the measures $m$ and $\P_m$ are only auxiliary tools to give this definition, and will not be used in the rest of the paper. Since the group $G$ acts on sample paths and it preserves equivalence classes, the action of $G$ extends to a natural action on the Poisson boundary. Moreover, the definition above implies that $\nu$ is \emph{$\mu$-stationary}, that is 
$$
\mu*\nu=\sum_g \mu(g) g\nu = \nu.
$$

Bounded $\mu$-harmonic functions can be represented by the Poisson boundary. A real-valued function $f$ on $G$ is called \emph{$\mu$-harmonic} if it satisfies the mean value property:
$$
P^\mu f(g): = \sum_{h \in G} f(g h) \mu(h)= f(g) \qquad \textup{for any }g \in G.
$$
Let $H^\infty(G,\mu)$ denote the space of all bounded, $\mu$-harmonic functions. 
Then the Poisson formula \cite{Furstenberg-semi} states that 
the  map 
$$H^\infty(G,\mu) \to  L^\infty(\partial_\mu G, \nu) $$
$$ f \mapsto \hat f $$
defined by $\hat f (\bm{bnd}(\bm w) )= \displaystyle\lim_{n\to\infty} f(w_n)$
is an isometric isomorphism, with inverse given by $f(g)=\int_{\partial_\mu G} \hat f \ dg\nu$.  

\begin{defin}
A $\mu$-boundary for $(G, \mu)$ is a measurable $G$-space $(B, \lambda)$ which is the quotient of the Poisson boundary with respect to a $G$-equivariant measurable partition.
\end{defin}

Similarly to the construction of the Poisson boundary, there exists a natural measurable map, called boundary map, $\pi_B: \Omega \to B$ such that the image of $\P$ under $\pi_B$ is $\lambda$. By the definition above, the Poisson boundary is the maximal $\mu$-boundary. 

Now, suppose that $G$ acts by isometries on a metric space $X$ with a base point $o$, and almost every sample path $(w_n o)$ converges 
to a point in a suitable topological compactification $\partial X$ of $X$. Then one defines the \emph{hitting measure} $\lambda$ on $\partial X$ by setting for any Borel $A \subseteq \partial X$ 
$$\lambda(A) := \mathbb{P}(\lim_{n \to \infty} w_n o \in A).$$
Then $(\partial X, \lambda)$ is a $\mu$-boundary. In order to prove it is also the Poisson boundary, we need to show 
it is maximal. The main tool to prove maximality of a $\mu$-boundary is the entropy theory.

\subsection{Entropy} 
Consider a $\mu$-random walk on a countable group $G$.  Let $(\Omega,\P)$ be the space of sample paths associated to the $\mu$-random walk, starting at the identity element of $G$.  
We shall use the language of partitions to define entropy. Let $\rho=(\rho_i)$ be a countable partition on the space of sample paths. Denote the entropy of $\rho$ as
$$
H_\P(\rho)=H(\rho) := - \sum_i  \P(\rho_i) \log\P(\rho_i). 
$$ 
Two sample paths $\bm w=(w_i)$ and $\bm w'=(w'_i)$ are $\mathcal{A}_n$-equivalent when $w_n=w'_n$.  Note that the distribution of $w_n$ is $\mu^{*n}$, the $n$-fold convolution of $\mu$, hence  
$$
H(\mu^{*n}) = H(\mathcal{A}_n) =- \int_\Omega \log \P(w_n) d\P(\bm w)  = - \sum_g \mu^{*n}(g) \log\mu^{*n}(g).
$$

\subsection{Conditional Entropy}\label{sec:conditional}
Let $(B, \lambda)$ be a $\mu$-boundary, with boundary map $\pi_B: (\Omega, \P) \to (B, \lambda)$. Then for $\lambda$-almost every $\xi$ in $B$, the conditional probability measure $\P^\xi$ exists and 
$$
\P= \int _B\P^\xi \ d\lambda(\xi). 
$$
Let $\rho=(\rho_i)$ be a countable partition,  define  the \emph{conditional entropy} given $\xi \in B$ as 
$$
H_\xi(\rho) := H_{\P^\xi}(\rho) =  - \sum_i  \P^\xi(\rho_i) \log\P^\xi(\rho_i). 
$$
We define the conditional entropy of $\rho$ given the boundary $B$ as 
$$H_B(\rho) := \int_B H_\xi(\rho) \ d\lambda(\xi).$$
Note that by convexity we have $H_B(\rho) \leq H(\rho)$ for any $\mu$-boundary $(B, \lambda)$.

We use the following lemma to determine whether a $\mu$-boundary is the Poisson boundary. 

\begin{thm}\cite[Theorem 4.5]{BG-semi} \label{T:criterion}
Let $(B, \lambda)$ be a $\mu$-boundary. If $H_B(\mathcal{A}_1) $ is finite, then there exists the limit 
$$
h(B,\lambda) :=\lim_{n\to\infty} \frac{H_B(\mathcal{A}_n)}{n}.
$$
Moreover, $h(B,\lambda) =0$ if and only if $(B,\lambda)$ is the Poisson boundary. 
\end{thm}

This formulation is stronger than 
\cite[Theorem 4.6]{Ka}, as it does not require the finite entropy of $\mu$. 
Note that Theorem~ \ref{T:criterion} is a consequence of the zero-entropy criterion for the Poisson boundary of random walks in random environment \cite{VK-RWRE}, 
also a consequence of the zero-entropy criterion for Markov chains on equivalence relations \cite[Theorem 2.17]{Kaimanovich-Sobieczky2012}; for random walks on groups, the equivalence relations consist of orbits of $\mu$-boundary points, see \cite{BG-semi} for more details.

\subsection{A general lemma about entropy} 

The following lemma will be crucial for our entropy arguments. 

\begin{lem} \label{L:entropy-restricted}
Let $A$ be a countable set, and $a \in A$. 
Let $Z : (\Omega, \mathbb{P}) \to A \setminus\{a\}$ be a random variable, with $H(Z) < + \infty$. 
Then for any $\epsilon > 0$ there exists $\delta>0$ such that, for any measurable set $E \subseteq \Omega$ such that $\mathbb{P}(E) < \delta$, 
the random variable 
$$Y(\omega) := \left\{ \begin{array}{ll}  Z(\omega) & \textup{if } \omega \in E\\
							 a & \textup{if } \omega \notin E 
							 \end{array} \right.$$
satisfies $H(Y) < \epsilon$. 
\end{lem}

\begin{proof}
Define $\varphi(t) := - t \log t$, which is increasing on $[0, 1/e]$ and decreasing on $[1/e, 1]$. 
Since $Z$ has finite entropy, there exists a finite subset $U \subseteq A \setminus\{a\}$ such that 
$$\sum_{x \notin U} \varphi(\mathbb{P}(Z = x)) < \epsilon/2.$$
Let us also pick $\delta < 1/e$. Hence, if $\mathbb{P}(E) < \delta$, we obtain
\begin{align*}
H(Y) & = \varphi( \mathbb{P}(E^c)) +  \sum_{x \in A} \varphi(\mathbb{P}( \{Z = x\} \cap E) )  \\ 
& = \varphi( \mathbb{P}(E^c)) +  \sum_{x \in U} \varphi(\mathbb{P}( \{Z = x\} \cap E) ) +  \sum_{x \in U^c} \varphi(\mathbb{P}( \{Z = x\} \cap E) )  \\
& \leq \varphi( \mathbb{P}(E^c)) +  \sum_{x \in U} \varphi( \mathbb{P}(E)) +  \sum_{x \in U^c} \varphi( \mathbb{P}(\{Z = x\}) )  \\
& \leq \varphi(1-\delta) + \# U  \cdot \varphi(\delta)  + \epsilon/2.
\end{align*}
Now, we can choose $\delta$ small enough so that $\varphi(1-\delta) + \# U  \cdot \varphi(\delta) < \epsilon/2$, completing the claim.
 \end{proof}
 
\begin{rem}
In the previous proof, the value of $\delta$ depends only on $\epsilon$ and the law of $Z$; this fact will play an important role 
in our application of Lemma \ref{L:entropy-restricted}.
\end{rem}

\section{Hyperbolic Geometry and Pivots}\label{sec:hyper}

\subsection{Pivots: the free group case}

The notion of pivots is essential in our construction of the pin down partition. For the sake of exposition, let us first define pivots in the simpler case of a free group; 
in this case, this notion does not require any knowledge of hyperbolic geometry. 

Let $G = F_k$ be a free group of finite rank, and consider its action on its own Cayley graph for the standard generating set.
Such a Cayley graph is a tree of valence $2 k$, which we will denote as $T$. We take as a base point $o \in T$ the identity element.
We also fix a constant $M \geq 0$. 

Given a point $p \in T$, we define the \emph{shadow} $S_{p}$ centered at $p$ as the set of points $x \in T$ such that the geodesic joining $x$ and $o$ passes within distance $ M $ of the point $p$. 

Let us now consider the sample path $(w_n)$ for a random walk on $G$. 

\begin{defin}
	Given an integer $n \geq 1$, we define the set $P_n$ of \emph{pivotal times} as follows. 
An integer $k \in \{1, \dots, n \}$ is a pivotal time if for any $h$ with $k \leq h \leq n$, the point $w_h$ lies in the shadow $S_{w_k}$ of $w_k$. 
\end{defin}

\begin{defin}
We say a time $k$ is \emph{pivotal from infinity} if, for any $n \geq k$, the point $w_n$ lies in the shadow $S_{w_k}$. 
The point $w_k$ is called a \emph{pivotal element}. 
\end{defin}

Note that, if we denote as $\xi \in \partial T$ the limit point of the sample path $(w_n)$, then for any pivotal time from infinity $k$, the element $w_k$ lies within distance $ M $ of $[o, \xi)$. 

Let us point out that, for a random walk on a free semigroup, \emph{every} time is a pivotal time. Of course this is no longer true passing from the free semigroup 
to the free group; however, we shall show that pivotal times are somehow abundant. 

In the general case of a group $ G $ acting on a hyperbolic space $ X $, the constant $ M $ also depends on the geometry of $ X $. In the general case, not only does $ w _{k} $ lie within distance $ M $ of $ [o, \xi) $, but so does $ w _{k} s $ for one choice among finitely many group elements $ s \in G $. This allows us to apply our argument, for instance, to $ F _{\infty} $.

To explain the notion of pivots carefully in this case, we need some background on hyperbolic geometry. 

\subsection{Hyperbolic spaces}

Let $(X, d)$ be a  metric space and $o \in X$ a base point. For $x,y$ and $z$ in $X$, define the \emph{Gromov product} as 
$
(x,y)_z := \frac{1}{2} \Big(d(x,z) + d(y,z) - d(x,y) \Big).
$
Let $\delta \geq 0$. The metric space $(X,d)$ is called \emph{$\delta $-hyperbolic} if for every $x,y,z$, and $o$ in $X$
$$
(x, y)_{o} \geq \min \big\{ (x, z)_{o}, (y, z)_{o} \big\} - \delta.
$$

We say that $ (X,d) $ is \emph{hyperbolic} if it is $ \delta$-hyperbolic for some $ \delta \geq 0 $. A group $ G $ is hyperbolic if it is generated by some finite set $ \Sigma $ and the Cayley graph $ \Gamma(G, \Sigma) $ is hyperbolic. By increasing $ \delta $ if necessary, we can also ensure that all triangles are $ \delta $-thin.

\begin{defin} 
Given constants $C, D \geq 0$, a sequence of points $ x _{0}, ..., x _{n} $ is a $ (C,D)$-chain if $ (x _{i-1}, x _{i+1}) _{x _{i}} \leq C $ for all $ 0 < i < n $ and $ d(x _{i}, x _{i+1}) \geq D $ for all $ 0 \leq i < n $.
\end{defin}

The following ``local-to-global principle" for geodesics is well known. 

\begin{lem}[\cite{pivot}, Lemma 3.7, 3.8] \label{L:canoe}
Let $ x _{0}, ..., x _{n} $ be a $ (C,D) $-chain with $ D \geq 2C + 2 \delta + 1 $. Then $ (x _{0}, x _{n}) _{x _{i}} \leq C+2\delta $ and \[ d(x _{0}, x _{n}) \geq n .\] 
\end{lem}

\begin{defin} 
Let $ y, y ^{+}, z \in X $. We say that $ z $ belongs to the \emph{$ C$-chain-shadow} of $ y ^{+} $ seen from $ y $ if there exists a $ (C, 2C+2\delta+1) $-chain $ x _{0} = y, x _{1}, ..., x _{n} = z $ satisfying $ (x _{0}, x _{1}) _{y ^{+}} \leq C $. 
\end{defin}


\subsection{Hyperbolic boundary}
A sequence $(z_n)$ in $X$ is \emph{divergent} if $\displaystyle\liminf_{n,k \to \infty }(z_n,z_k)_o =\infty$ for some $o$ in $X$.  Two divergent sequences $(z_n)$ and $(z'_n)$ are $\simeq$-equivalent when  $\displaystyle\lim_{n \to \infty }(z_n,z'_n)_o =\infty$.  
We denote by $\partial X$ the \emph{hyperbolic boundary} of $X$, that is the quotient of the set of divergent sequences with respect to the equivalence relation $\simeq$. 
If the space $X$ is proper and geodesic, the boundary $\partial X$ can be also described as the space of geodesic rays $\gamma : [0, \infty) \to X$ 
with $\gamma(0) = o$, modulo the equivalence relation that identifies two geodesic rays $\gamma_1(t), \gamma_2(t)$ if 
$\sup_{t > 0} d(\gamma_1(t), \gamma_2(t)) < \infty$. 


Let $G$ be a group acting by isometries on a hyperbolic metric space $(X, d)$, with a base point $o \in X$. 
Recall that an element $g$ is 
\emph{loxodromic} if its translation length $\tau(g) := \lim_{n \to \infty} \frac{d(g^n o, o)}{n}$ is positive. 
A semigroup acting by isometries on a hyperbolic metric space is \emph{non-elementary} if it contains two 
loxodromic elements with disjoint fixed sets on $\partial X$. A measure on $G$ is non-elementary 
if the semigroup generated by its support is non-elementary. 
 
\subsection{Schottky sets}

The notion of \emph{Schottky set} arises from \cite{Sisto} and is used in \cite{pivot} to define pivots. 

\begin{defin} 
	Let $ \varepsilon, C, D \geq 0 $. A finite set $ S \subset \text{Isom}(X) $ is $ (\varepsilon, C, D)-$Schottky if \begin{enumerate} 
	\item For all $ x,y \in X $ we have $ \# \{ s \in S| (x, sy) _{o} \leq C\}  \geq (1- \varepsilon) \# S  $.
	\item  For all $ x,y \in X $ we have $ \# \{ s \in S| (x, s ^{-1} y) _{o} \leq C \}  \geq (1- \varepsilon) \#S $.
	\item For all $ s \in S $, we have $ d(o, s o) \geq D $.
	  \end{enumerate}
\end{defin}
		   
Recall the following fact from \cite{pivot}:

\begin{lem}\label{L:Schottky}\cite[Corollary 3.13]{pivot}
Let $ \mu $ be a non-elementary measure on a countable set of isometries of a $ \delta $-hyperbolic space $ X $. For all $ \varepsilon > 0 $, there exists $ C > 0 $ such that for all $ D > 0 $, there exist 
$ \ell > 0 $ and an $ ( \varepsilon, C, D)-$Schottky set in the support of $ \mu^{*\ell} $.   
\end{lem}

Then by setting $N =  2 \ell$ we can write the decomposition $ \mu ^{*N} = \alpha \mu _{S} ^{*2} + (1- \alpha) \uptau $ where $ \mu _{S} $ is the uniform measure on some $ (\varepsilon, C, D) $-Schottky set. We choose $ D $ sufficiently large depending on $ \delta $ and $ C $ (for instance, $ D > 100C+100\delta+1$ suffices).
	  
\subsection{Alternating random walks and stopping times}	  
	  
\begin{prop} \label{P:alternate}
Let $\mu$ be a non-elementary probability measure on $G$,
a countable set of isometries of a Gromov-hyperbolic space. For any $ \varepsilon > 0 $, there exists $ C > 0 $ such that for any $ D > 0 $, there exists a probability measure $\theta$ on $G$ with the following properties:
\begin{enumerate}
	\item There exists a $ (\varepsilon, C, D)$-Schottky set $S$ such that $\theta = \kappa * \mu^{*2}_S$, where $\mu_S$ is the uniform probability measure on $S$. 
\item The Poisson boundary of $\theta$ is the same as the one of $\mu$.
\item If $\mu$ has finite entropy, then the entropy of $\theta$ is finite.
\end{enumerate}
\end{prop}
\begin{proof}
	Consider the decomposition from our remark after Lemma \ref{L:Schottky}:
$\mu^{*N} = a \mu^{*2}_S + (1-a) \uptau$ for some probability measure $\uptau$, $N$, and $0 < a < 1$.  Define 
$$
\kappa := a\sum_{n=0}^\infty (1-a)^n \uptau^{*n}.
$$
Part (2) and (3) follow from a stopping time trick (we stop the random walk when we select an increment from $S^2$) \cite{Forghani, BK2013}. 
In particular, part (2) follows from \cite[Proposition 6.3]{BG-semi} and part (3) follows from \cite[Theorem 2.4]{Forghani}.
\end{proof}

From now on, we shall call 
a measure of type 
$$\theta := \kappa * \mu_S^{*2},$$ 
where $\mu_S$ is the uniform measure on some Schottky set $S$, and $\kappa$ 
any probability measure on $G$, an \emph{alternating measure}. 

\subsection{Pivots: general definition}

We now define pivotal times for general group actions on hyperbolic spaces, following \cite[Section 4]{pivot}. 
Let $\omega \in \Omega$ be a sample path, and $n \geq 0$ an integer. 
First we set $ P _{0} := \varnothing $. Given $ P _{n-1} $, we define $ P _{n} $ inductively as follows. Let $ k = k(n) $ be the last pivotal time before $ n $. Write the random walk induced by $ \theta = \kappa * \mu _{S} ^{*2} $ as $ w _{0} s _{1} ... w _{n-1} s _{n} w _{n} $ where $ w _{i} \sim \kappa $ and $ s _{i} = a _{i} b _{i} $ for $ a _{i}, b _{i} \sim \mu _{S} $. 
We set 
\begin{align*} 	    
	y _{i} ^{-} &:= w _{0} s _{1} ... s _{i-1} w _{i-1} \\ 
	  y _{i} &:= y _{i} ^{-} a _{i}, \\
  y ^{+} _{i} &:= y _{i} ^{-} a _{i} b _{i} .
\end{align*}

We say that the \emph{local geodesic condition} is satisfied at time $n $ if 
$$ (y _{k}, y _{n}) _{y _{n} ^{-}} \leq C, \qquad (y _{n} ^{-}, y _{n} ^{+}) _{y _{n}} \leq C, \qquad \textup{ and }(y _{n}, y _{n+1} ^{-}) _{y _{n} ^{+}} \leq C.$$ 
If the local geodesic condition is satisfied at time $ n $, we set $P _{n} := P _{n-1} \cup \{ n \} $. Else, we backtrack to the largest pivotal time $ m \in P _{n-1} $ such that $ y _{n+1} ^{-} $ is in the $ (C+\delta) $-chain-shadow of $ y _{m} ^{+} $ seen from $ y _{m} $. In other words, we set $ P _{n} := P _{n-1} \cap \{ 1, ..., m \} $.

Recall the following lemma  \cite[Lemma 4.4]{pivot}.

\begin{lem} \label{L:Chain}
Let $ P _{n} = \{ k _{1} < ... < k _{p} \} $. Then the sequence $ e, y _{k _{1}}, y _{k _{2}} ^{-}, y_{k_2}, \ldots, y _{k _{p}}, y _{n+1} ^{-} $ is a $ (2C+4\delta, D-2C-3\delta) $-chain. 
\end{lem}

This combined with Lemma \ref{L:canoe} implies that  
$ (e, y _{n+1} ^{-}) _{y _{k _{i}}^-} \leq 2C+6\delta $  and $ (e, y _{n+1} ^{-}) _{y _{k _{i}}} \leq 2C+6\delta $. In other words, both $y_{k_i}^-$ and $ y _{k _{i}} $ lie within a uniformly bounded neighbourhood of the geodesic $ [e, y _{n+1} ^{-}] $.		

\begin{defin}
We say that a time $ k $ is \emph{pivotal from infinity} if $ k \in P _{m} $ for all $ m \geq k$.  
\end{defin}

If $ k _{i} $ is a pivotal time from infinity and $\xi := \lim_{i \to \infty} y_{k_i} \in \partial G$ exists, then $ y _{k _{i}} $ lies within a uniformly bounded neighbourhood of any geodesic connecting $ e $ and $ \xi \in \partial G $; that is, there exists a constant  $M := 2 C + 9 \delta$, which does not depend on $\xi$, such that, for any geodesic $\gamma$ between $e$ and $\xi \in \partial G$ and any pivot 
from infinity $y_{k_i}$, we have 
\begin{equation} \label{E:fellow-t}
\max \{  d(y_{k_{i}}^-, \gamma), \ d(y_{k_{i}}, \gamma) \} \leq M. 
\end{equation}

\section{The Entropy Argument}\label{sec:pin-down}

\subsection{Defining the partitions}

Let $G$ be a non-elementary hyperbolic group, and fix a choice of finite generating set $ \Sigma $. Let $ X = \Gamma (G, \Sigma)$ be the corresponding Cayley graph and fix a geodesic $ \xi $ connecting $ e $ to $ \partial G$. Let us consider an alternating measure $\theta = \kappa  * \mu_S^2$ on $G$, and let $(w_n)$ be a sample path for a random walk driven by $\theta$. 

Let $\alpha$ be a constant to be fixed later.  All the following quantities will depend on an infinite sample path $\omega \in \Omega$.

\begin{defin}
Let $n \geq 1$. For each $0 \leq k \leq \lceil \frac{n}{\alpha} \rceil$, let us denote the time interval 
$$I_{k, \alpha} := (\alpha (k-1), \alpha k] \cap [0, n].$$ 
Moreover, let $t_k$ be the time of the first pivot from infinity inside $I_{k, \alpha}$, 
if one such pivot exists; in this case, let us denote as 
$$p_k := w_{t_k}$$ the pivotal element; 
for $k = 0$, let $p_0 := e$; if $k > 0$ and there is no pivot inside $I_{k, \alpha}$, let us set $t_k := -1$.
\end{defin}

We denote as $\mathcal{T}_n^\alpha := (t_1, \dots, t_{n/\alpha})$ and we call it the set of \emph{pivotal times}.  

\medskip
Let $L$ be a constant to be fixed later. 

\begin{defin}
We say an interval $I_{k, \alpha}$ for $1 \leq k < \frac{n}{\alpha}$ is \emph{$L$-good} if there is a pivot in $I_{k, \alpha}$, and if
$$d(w_{j-1}, w_{j}) \leq L \qquad \textup{for all } j \in I_{k, \alpha}.$$
Finally, by construction we declare $I_{0, \alpha} = \{ 0 \}$ to be $L$-good.
Otherwise, the interval is $L$-\emph{bad}. 
Note that by construction we define the last interval, that is the one containing $n$, to be $L$-bad. 
\end{defin}

Let us define the \emph{good distance} as the sum of the distances between consecutive good pivots:
$$D_n^{\alpha,L} := \sum_{\stackrel{I_{k, \alpha}, I_{k+1, \alpha} }{\textup{ $L$-good}}} d(p_k, p_{k+1})$$
where the sum is over all $k$ with $0 \leq k < \lceil n/\alpha \rceil$ such that both $I_{k, \alpha}$ and $I_{k+1,\alpha}$ are $L$-good intervals, and $d$ is the word length.

\begin{defin}
For each $L$-bad interval $I_{k, \alpha}$, let $J_{k ,\alpha} := (I_{k-1, \alpha} \cup I_{k, \alpha} \cup I_{k+1, \alpha}) \cap [0, n]$ and define as $b_k$ the finite sequence of group elements
$$b_k := (g_i)_{i \in J_{k, \alpha}}.$$
Let us define the $(\alpha,L)$-\emph{bad intervals} $\mathcal{B}_n^{\alpha,L} $ as the sequence
$$\mathcal{B}_n^{\alpha,L}  := ( b_{k_1}, \dots, b_{k_s}),$$
where $I_{k_1, \alpha}, \dots, I_{k_s, \alpha}$ are the $L$-bad intervals up to time $n$. 
\end{defin}

The key property of bad intervals is that they are somewhat rare, since pivots are abundant. 
The following lemma makes this precise. 

\begin{lem} \label{L:bad-are-rare}
Consider the random walk driven by an alternating measure $\theta$ on a non-elementary hyperbolic group $G$. 
For any $\delta > 0$ there exists $\alpha_0 > 0$ such that for any $\alpha \geq \alpha_0$ 
there exists $L \geq 1$ for which we have
$$\mathbb{P}\Big(I_{k, \alpha} \textup{ is $L$-bad}\Big) \leq \delta \qquad \textup{for any }n \geq 1, \  \textup{for any }k < \left \lceil \frac{n}{\alpha} \right \rceil.$$
\end{lem}

\begin{proof}
Recall that, for any $n, j, k \geq 0$, Gou\"ezel's paper \cite[Lemma 4.9]{pivot} shows 
\begin{equation} \label{E:pivots}
\mathbb{P}\Big(\#P_{n+j} \geq \# P_{n} + k \Big) \geq \mathbb{P}\left( \sum_{i = 1}^j U_i \geq k \right)
\end{equation}
where $(U_i)$ are i.i.d. random variables with $\mathbb{E}[U_1] > 0$ and  $\mathbb{E}[e^{- t_0 U_1}] < \infty$ for some $t_0 > 0$. Hence, there exists $t > 0$ with $\mathbb{E}[e^{- t U_1}] < 1$.
Standard large deviation estimates yield
$$ \mathbb{P}\left(\sum_{i = 1}^\beta U_i  \leq 0 \right) = \mathbb{P} \left(e^{-t \sum_{i = 1}^\beta U_i} \geq  1 \right) \leq  \mathbb{E}[e^{-t \sum_{i = 1}^\beta U_i}] =  (\mathbb{E}[e^{-t U_1}])^{\beta}.$$
Hence, we can apply it to our situation and obtain, setting $m = \alpha(k-1)$, 
\begin{align*}
\mathbb{P}\Big(\textup{there is no pivot from }\infty\textup{ in }I_{k, \alpha} \Big) & \leq \mathbb{P} \Big(\exists \beta \geq \alpha \ : \ \#P_{m+\beta} \leq \# P_{m}  \Big) \\
& \leq \sum_{\beta  = \alpha}^\infty  \mathbb{P}\Big(\#P_{m +\beta} \leq \# P_{m} \Big) \\
& \leq \sum_{\beta = \alpha}^\infty \mathbb{P}\Big(\sum_{i = 1}^\beta U_i  \leq 0 \Big) \\
& \leq \sum_{\beta = \alpha}^\infty  (\mathbb{E}[e^{-t U_1}])^{\beta} =  \frac{ (\mathbb{E}[e^{-t U_1}])^{\alpha}}{1 - \mathbb{E}[e^{-t U_1}]}  < \delta/2
\end{align*}
if $\alpha$ is large enough. 
Moreover, one can choose $L$ large enough so that 
$$\mathbb{P}\Big(d(e, g_1) \geq L \Big) \leq \frac{\delta}{2 \alpha}$$
hence, since the $(g_i)$ are i.i.d., 
\begin{align*}
\mathbb{P}\Big(d(w_{j-1}, w_{j})  \geq L \textup{ for some }j \in I_{k, \alpha} \Big) & \leq \sum_{j = 1}^\alpha \mathbb{P}\Big(d(w_{j-1}, w_{j}) \geq L \Big) \\
& \leq \sum_{j = 1}^\alpha \mathbb{P}\Big(d(e, g_{1}) \geq L\Big) \\
& \leq \alpha \cdot  \frac{\delta}{2 \alpha} = \frac{\delta}{2}
\end{align*}
hence, combining the two statements, we get the claim. 
\end{proof}

Given two partitions $\mathcal{A}$ and $\mathcal{B}$, we denote by $\mathcal{A} \vee \mathcal{B} := \{ a \cap b \ : \ a \in \mathcal{A}, b \in \mathcal{B} \}$ their join. When $\mathcal{A}=(A_i)$ and $\mathcal{B}=(B_i)$ are two measurable countable partitions on the path space, define $\P^\xi( A_i | B_j) :=\frac{\P^\xi(A_i \cap B_j)}{\P^\xi (B_j)}$. Then, we define  the associated conditional entropy as
$$
H_\xi(\mathcal A | \mathcal B) := - \sum_{i,j} \P^\xi (A_i \cap B_j)  \log \P^\xi(A_i| B_j), \mbox{ and } H_B(\mathcal A | \mathcal B ) := \int_B H_\xi(\mathcal A | \mathcal B ) \ d\lambda(\xi). 
$$

\begin{prop} \label{P:pin-down}
Let $G$ be a hyperbolic group, and $\partial G$ its hyperbolic boundary. Then for any $\alpha \geq 1, L \geq 1$ the join of the partitions $\mathcal{T}_n ^\alpha$, $ D_n^{\alpha,L}$, and $\mathcal{B}_n^{\alpha,L} $ pins down the location $w_n$ of the walk, in the following sense: 
\begin{equation} \label{E1-1}
H_{\partial G}(\mathcal{A}_n \vert  \mathcal{T}_n^\alpha \vee  D_n^{\alpha,L} \vee \mathcal{B}_n^{\alpha,L}) = o(n).
\end{equation}
\end{prop}

\begin{proof}
Let us fix a boundary point $\xi \in \partial G$, and let us split the set $[0, \lceil n/\alpha \rceil ]$ of indices $k$ for $I_{k, \alpha}$ as a disjoint union of indices of $L$-good intervals and maximal chains of consecutive indices of $L$-bad intervals. 
Let $k_i, \dots, k_{i+r}$ be a maximal chain of consecutive indices of $L$-bad intervals, except for the last one (that is, so that 
$I_{k_{i+r}, \alpha}$ does not contain $n$).
Then $t^- := t_{k_{i-1}}$ and $t^+ := t_{k_{i+r+1}}$ are pivotal times, and  $p_-:= w_{t_{k_{i-1}}}$ and $p_+:= w_{t_{k_{i+r+1}}}$ are pivotal elements.
Consider the word 
$$W_i := g_{t^-+1}  \cdot \ldots \cdot g_n \cdot \ldots \cdot g_{t^+}.$$
Record $d_i := d(e,  W_i)$ the word length of $W_i$. 
Now, let us pick a geodesic ray $\gamma = [e, \xi)$ connecting the origin with $\xi$. Since every pivot lies in the $M$-neighborhood of $\gamma$ and there are at most $n/\alpha$ pivots, the sum 
$$ D_n^{\alpha,L} + \sum_{i \textup{ $L$-bad}} d_i$$
yields the distance from the origin of the last pivot $p_{last}$ before $w_n$, up to an additive error of at most $2 M n /\alpha$.
 
Recall the constant $M$, which depends only on $C$ and $\delta$, from Eq. \eqref{E:fellow-t}. Since by definition $p_{last}$ lies in an $M$-neighborhood of $\gamma$,
the point $p_{last}$ lies in the union of $\frac{4 n}{\alpha}$ balls of radius $2 M$ in $G$, hence there are at most $\frac{4 n}{\alpha} \cdot \#B(2M)$ choices for $p_{last}$. 

By definition, all intervals $I_{k, \alpha}$ after the last pivot are $L$-bad, so 
we also record 
$$W_{last} := g_{t^-+1} \cdot \ \dots \  \cdot g_n$$
hence we have 
$$w_n = p_{last} W_{last}$$
pinned down up to at most $\frac{4 n}{\alpha} \cdot \#B(2M)$ choices, hence 
$$H_\xi(\mathcal{A}_n \vert  \mathcal{T}_n^\alpha \vee  D_n^{\alpha,L}  \vee \mathcal{B}_n^{\alpha,L} ) \leq \log ((4n/\alpha) \cdot \#B(2M)) = o(n)$$
which, after integrating over $\xi$, proves the claim.
\end{proof}

\begin{prop}
Let $G$ be a hyperbolic group, let $\partial G$ be its hyperbolic boundary, 
and consider a random walk driven by the alternating measure $\theta := \kappa * \mu_S^{*2}$. 
If $\theta$ has finite entropy, then
for any $\epsilon > 0$, there exists $\alpha_0 > 0$ such that for any $\alpha \geq \alpha_0$ there exists $L \geq 1$ such that
\begin{equation} \label{E2-1}
	\limsup_{n} \frac{H( \mathcal{T}_n^\alpha \vee D_n^{\alpha,L}  \vee \mathcal{B}_n^{\alpha,L} )}{n} \leq \epsilon.
\end{equation}
\end{prop}

\begin{proof}
Since each pivotal time $t_i$ has at most $\alpha + 1$ choices, and we record at most $\frac{n}{\alpha}$ pivotal times, we get 
$$H(\mathcal{T}_n^\alpha) \leq \frac{n}{\alpha} \log(\alpha + 1).$$

Moreover, the value of $ D_n^{\alpha,L}$ is an integer between $0$ and $Ln$, hence 
$$
H( D_n^{\alpha,L}) \leq \log(Ln + 1).
$$
Let us consider for each $k \leq \lceil n/\alpha \rceil$ the variable 
$$Y_k := (g_{\alpha (k-1)+1}, \dots, g_{\alpha(k+2)}) \mathbf{1}_{\{I_{k, \alpha} \textup{ is $L$-bad}\}}$$
 and for each $j \leq n$, $k \leq \lceil n/\alpha \rceil$ the variable 
$$Y_{j, k} := g_j \mathbf{1}_{\{I_{k, \alpha} \textup{ is $L$-bad}\}}.$$

Fix $\epsilon > 0$, and let $\delta > 0$ be given by Lemma \ref{L:entropy-restricted} applied to $Z = g_1$, which has finite entropy 
by assumption. 
Now, by Lemma \ref{L:bad-are-rare}, if $\alpha$ and $L$ are sufficiently large, then
$$\mathbb{P}(I_{k, \alpha} \textup{ is a $L$-bad interval}) \leq \delta \qquad \textup{for every }k < \lceil n / \alpha \rceil.$$ 
Thus, we can set $E_k := \{ I_{k, \alpha} \textup{ is a $L$-bad interval} \}$ and, noting that all $(g_j)$ for $1 \leq j \leq n$ have the same distribution, we get, by Lemma \ref{L:entropy-restricted}, 
$$H(Y_{j, k}) \leq \epsilon \qquad \textup{for any } k \leq \lceil n/ \alpha \rceil, \textup{ any }j \in I_{k, \alpha}.$$
Thus, by summing over all $j$, we obtain 
$$H(Y_k) \leq \sum_{j \in I_{k, \alpha}} H(Y_{j, k}) \leq 3 \epsilon \alpha \qquad \textup{for any }k < \lceil n/\alpha \rceil $$
and by summing over all $k$, 
$$H(\mathcal{B}_n^{\alpha,L} ) \leq \sum_{k \leq \lceil n/\alpha \rceil } H(Y_k) \leq 3\epsilon n +  3 \alpha H(\theta).$$
Hence 
\begin{align*}
	\limsup_n \frac{H( \mathcal{T}_n^\alpha \vee D_n^{\alpha,L}  \vee \mathcal{B}_n^{\alpha,L} )}{n} & \leq 
	\limsup_{n} \frac{H( \mathcal{T}_n^\alpha)}{n} +  \limsup_n \frac{H( D_n^{\alpha,L})}{n} + \limsup_n \frac{H(\mathcal{B}_n^{\alpha,L} )}{n}\\
												& \leq \frac{\log(\alpha+1)}{\alpha}  + 0 + 3 \epsilon
\end{align*}
so the $\limsup$ is bounded by $4 \epsilon$ if we choose $\alpha$ large enough. 
\end{proof}

\begin{thm} \label{T:low-entropy}
Let $G$ be a hyperbolic group, let $\partial G$ be its hyperbolic boundary, 
and consider a random walk driven by the alternating measure $\theta := \kappa * \mu_S^{*2}$. 
If $\theta$ has finite entropy,  then we have 
$$\lim_{n \to \infty} \frac{H_{\partial G}(\mathcal{A}_n )}{n} = 0.$$

\end{thm}

\begin{proof}
Fix $\epsilon > 0$. By \eqref{E1-1} and \eqref{E2-1}, we can choose $\alpha, L$ large enough so that, if we denote 
$\mathcal{P}_n := \mathcal{T}_n^\alpha \vee D_n^{\alpha,L}  \vee \mathcal{B}_n^{\alpha,L} $, we have 
\begin{align*}
\limsup_{n\to \infty} \frac{H_{\partial G}(\mathcal{A}_n )}{n} & \leq \limsup_{n\to \infty} \frac{H_{\partial G}(\mathcal{A}_n \vee \mathcal{P}_n )}{n} \\
& \leq \limsup_{n\to \infty} \frac{H_{\partial G}(\mathcal{P}_n )}{n} + \limsup_{n\to \infty} \frac{H_{\partial G}(\mathcal{A}_n  \vert \mathcal{P}_n )}{n} \\ 
& \leq \limsup_{n\to \infty}  \frac{H(\mathcal{P}_n )}{n} + \limsup_{n\to \infty} \frac{H_{\partial G}(\mathcal{A}_n  \vert \mathcal{P}_n )}{n} \leq \epsilon
\end{align*}
from which the claim follows.
\end{proof}

As a corollary, the space of asymptote classes of geodesic rays emanating from the identity is a model for the Poisson boundary. 

\begin{proof}[Proof of Theorem \ref{T:main}]
By Proposition \ref{P:alternate}, there is a decomposition $\mu^{*N} = a \mu_S^{*2} + (1-a) \uptau$, 
with $0 < a < 1$ and $\mu_S$ uniform on a Schottky set $S$. Consider the random walk 
driven by the alternating measure $\theta := \kappa * \mu_S^{*2}$. 
Since the support of $\theta$ is non-elementary, the random walk converges to the hyperbolic boundary $\partial G$; 
let $\lambda$ be the corresponding hitting measure. 
By Theorem \ref{T:low-entropy}, for the random walk $(G, \theta)$ we have
$$\lim_{n \to \infty} \frac{H_{\partial G}(\mathcal{A}_n )}{n} = 0$$
hence by Theorem \ref{T:criterion}, the space $(\partial G, \lambda)$ 
is the Poisson boundary.  Finally, by Proposition \ref{P:alternate}, the Poisson boundary of $(G, \mu)$ equals the Poisson boundary of $(G, \theta)$, hence 
it is also given by $(\partial G, \lambda)$.
 \end{proof}

 \section{Actions on hyperbolic spaces with WPD elements}
 
 Let $G$ be a group of isometries of a hyperbolic metric space $(X, d)$. 
 We now extend the previous proof to the more general case that $G$ only contains a WPD element. 
 This includes a very large class of non-hyperbolic groups, such as relatively hyperbolic groups, right-angled Artin groups, mapping class groups, and $Out(F_n)$.  
 
Given $K > 0$ and $x, y \in X$, we define the joint coarse stabilizer as 
$$\textup{Stab}_K(x, y) := \{ g \in G \ : \ d(x, gx) \leq K \textup{ and }d(y, gy) \leq K \}.$$

\begin{defin} \label{D:WPD}
An element $g \in G$ is \emph{WPD} if it is loxodromic and for any $x \in X$, any $K \geq 0$ there exists $P \geq 1$ such that 
$\textup{Stab}_K(x, g^P x)$ is finite.
\end{defin}

The main theorem of this section is the following, which is also Theorem \ref{T:WPD-action-intro} in the introduction. 

 \begin{thm}  \label{T:WPD-action}
Let $ G $ be a countable group with a non-elementary action on a geodesic hyperbolic space $ X $ with at least one WPD element. Let $ \mu $ be a generating probability measure on $ G $ with finite entropy. Let $ \partial X $ be the hyperbolic boundary of $ X $, and let $ \nu $ be the hitting measure of the random walk driven by $ \mu $. Then the space $ (\partial X, \nu) $ is the Poisson boundary of $ (G, \mu) $.
\end{thm}

Let us start with a simple geometric lemma in $\delta$-hyperbolic spaces. 

\begin{lem} \label{L:fellowtravel}
Let $X$ be a $\delta$-hyperbolic metric space. Let $a$ be a point which is within distance $C$ of a geodesic segment 
$\gamma$ connecting points $x, y \in X$. If $g \in Isom(X)$ satisfies $d(x,  gx),  d(y, gy) \leq K$, then 
$d(a, ga) \leq 2C + 2 K + 4 \delta$.
\end{lem}

\begin{proof}
Pick a point $ t \in [0,d(x,y)] $ such that $ d(a, \gamma (t)) \leq C $. Consider the quadrangle $ [x, gx] \cup [gx, gy] \cup [gy, y] \cup [y,x] $. Since $ X $ is $ \delta $-hyperbolic, then we can find some point $ z \in [x, gx] \cup [gx, gy] \cup [gy, y] $ such that $ d( \gamma(t),z) \leq 2 \delta $.

If $ z \in [x,gx] $, then we have 
\begin{align*} 
	d(a, ga) &\leq d(a, \gamma (t)) + d(\gamma (t), z) + d(z, gx) + d(gx, g \gamma (t)) + d(g \gamma (t), ga) \\
				   &\leq C + 2\delta + K + (2\delta + K) + C \\
				   &= 2 C  + 2K + 4\delta
\end{align*}
and likewise if $ z \in [gy,y] $.

Now suppose that $ z \in [gx, gy] $, so that $ z = g \gamma (t') $ for some $ t' \in [0, d(x,y)] $. Then we have 
\begin{align*} 
	t =  d(x, \gamma (t)) &\leq d(x, gx) + d(gx, g \gamma (t')) + d(g \gamma (t'), \gamma (t)) \\
						     &\leq K + t' + 2 \delta 
\end{align*}
so that $ t - t' \leq 2\delta + K $. Likewise, considering $ d(y, \gamma (t)) $, we have $ t' - t \leq 2\delta + K $.

Therefore 
\begin{align*} 
      d(a, ga) &\leq d(a, \gamma (t)) + d(\gamma (t), g \gamma (t')) + d(g\gamma (t'), g \gamma (t)) + d(g \gamma (t), ga) \\
					       &\leq C + 2 \delta + (2\delta + K) + C\\
					       &= 2C + K + 4\delta.
\end{align*}
Either way, we have $ d(a, ga) \leq 2C + 2K +4\delta$. 		
\end{proof}

Given a loxodromic element $g$, we denote 
$$V_C(g^+) := \{ x \in X \ : \ (x, g^+)_o \geq C \}$$ 
where $g^+ \in \partial X$ is the attracting fixed point of $g$.

By slightly modifying Gou\"ezel's proof of Proposition 3.12 \cite{pivot}, we get the following: 

\begin{lem} \label{L:Sch-improved}
Let $o \in X$ be a base point. For any $0 < \epsilon < 1$, there exists $C$ such that for any $D$ there is a finite set $S \subseteq G$ which is 
$(\epsilon, C, D)$-Schottky,  and moreover: 
\begin{enumerate}
\item 
the sets $V_C(s^+)$ for distinct values of $s \in S$ are disjoint; 

\item
the point  $s o$ belongs to $V_{C+\delta}(s^+)$ for each $s \in S$.

\end{enumerate}
\end{lem}

\begin{proof}
The proof is the same as the proof of \cite[Proposition 3.12]{pivot}, using a ping pong argument; the only modification consists in taking $p$ sufficiently large such that (2) holds for each $s = g_i^p$. 
\end{proof}

\begin{lem}\label{WPDSchottky}
Suppose that $G$ has a nonelementary action on a $ \delta $-hyperbolic space $ X $ by isometries, and contains a WPD element. 
Let $o \in X$. For any $0 < \eta < 1$, there exists $C > 0$ such that for any $D > 0$ and for any $ K > 0 $ there is a finite set $S$ which is $(\eta, C, D)$-Schottky, 
and for any $s \in S$ the coarse joint stabilizer 
$\textup{Stab}_K(o, s o)$ is finite.
\end{lem}

\begin{proof}
Note that a Schottky set for some $D$ is also a Schottky set for any smaller $D$, so we can assume that $D > 2 C + a \delta$, where 
$a$ is a fixed constant whose value can be computed by looking at the details of the following proof. To avoid keeping track of it, 
we write $O(\delta)$ to mean a constant that only depends on the hyperbolicity constant $\delta$ of $X$.

Let $w$ be a WPD element. By the WPD property, there exists some $P \geq 1$ and $h = w^P$ such that 
$$d(o, h o) \geq D$$ 
and 
$$\#\text{Stab}_{K_1}(o, h o) < \infty$$
with $K_1 = 2 K + 2 C + 12 \delta$.

By Lemma \ref{L:Sch-improved}, we can find a finite set $S$ which is $(\epsilon, C, D)$-Schottky and satisfies (1) and (2). 
Let us consider the set 
$$S_1 := \{ s_1 h s_2 \ : \ s_1 \in S, s_2 \in S, (s_1^{-1} o, h o)_{o} \leq C, (h^{-1} o, s_2 o)_{o} \leq C  \}.$$
First, we claim that, if we choose $D$ large enough, any $s = s_1 h s_2 \in S_1$ satisfies
$$d(o, s_1 h s_2 o) \geq D.$$
In fact, the sequence $o, s_1 o, s_1 h o, s_1 h s_2 o$ is a $(C, D)$-chain, hence 
by Lemma \ref{L:canoe} for any geodesic segment $\gamma$ joining $o$ and $s_1 h s_2 o$, we obtain
$$\max \{ d(s_1 o,  \gamma), d(s_1 h o,  \gamma)\}  \leq C + 4 \delta.$$
Moreover, by looking at the corresponding approximate tree we obtain 
\begin{align*}
d(o, s_1 h s_2 o) & =  d(o, s_1 o) + d(s_1 o, s_1 h o) + d(s_1 h o, s_1 h s_2 o) - (o, s_1 h o)_{s_1 o} - (s_1 o, s_1 h s_2 o)_{s_1 h o} + O(\delta) \\
&  \geq 3 D - 4 C + O(\delta) \geq D
\end{align*}
if $D \geq 2 C + O(\delta)$. 

Second, we claim that the cardinality of $S_1$ satisfies 
\begin{equation} \label{E:cardinality}
\# S_1 \geq (1- \epsilon)^2  (\#S)^2.
\end{equation}
Note that by definition of Schottky set, $(o, s_1 h o)_{s_1 o} = (s_1^{-1} o, h o)_{o} \leq C$ for at least $(1- \epsilon)  \# S$ values of $s_1$. 
Similarly $(s_1 o, s_1 h s_2 o)_{s_1 h o} = (h^{-1} o, s_2 o)_{o}\leq C$ for at least $(1- \epsilon) \# S $ values of $s_2$. 

Now, fix $s_1$ and suppose that there exists $s_2, t_1, t_2$ in $S$ such that $s_1 h s_2 = t_1 h t_2$. We want to show that $s_ 1= t_1$, 
$s_2 = t_2$, which yields \eqref{E:cardinality}.
Then by $\delta$-hyperbolicity we have, as above, that both $s_1 o$ and $t_1 o$ lie with distance at most $C +  4 \delta $ of any geodesic $\gamma$
between $o$ and $s_1 h s_2 o = t_1 h t_2 o$. 
Thus, both geodesic segments $[o, s_1 o]$ and $[o, t_1 o]$ have length at least $D$ and fellow travel $\gamma$ for a distance at least $D - C - O(\delta)$, which implies that
$$(s_1 o, t_1 o)_o \geq D - C - O(\delta) \geq C + \delta$$
as soon as $D$ is large enough.  
This also implies that 
$$(t_1 o, s_1^+)_o \geq \min \{ (s_1^+, s_1 o)_o, (s_1 o, t_1 o)_o \} - \delta \geq C$$
hence $t_1 o \in V_C(s_1)$. 
However, since $(t_1 o, t_1^+)_o \geq C$ by property (2) and definition of $V_C$, we have 
$$t_1 o \in V_C(t_1^+) \cap V_C(s_1^+)$$
which implies $t_1 = s_1$ since the $V_C(s^+)$ for distinct $s \in S$ are disjoint. 
Then the relation $s_1 h s_2 = t_1 h t_2$ also implies $s_2 = t_2$. 

Now, we prove that for any $x, y \in X$
$$\# \{ s \in S_1 \ : \ (x, s ^{\pm} y)_o \leq C \} \geq (1-\epsilon) ) (\# S )^2.$$

Let $x, y \in X$. Then by definition of Schottky set, for any $s_2 \in S$ the number of $s_1 \in S$ such that  
$$(x, s_1 h s_2 y)_o \geq C$$
is at most $\epsilon  \# S$. Hence, the number of $s = s_1 h s_2 \in S_1$ such that $(x, s y)_o \geq C$
is at most 
$$\epsilon (\# S)^2 \leq \frac{\epsilon}{(1-\epsilon)^2} \# S_1.$$

Also by definition of Schottky set, for any $s_1 \in S$ the number of $s_2 \in S$ such that  
$$(x, s_2 ^{-1} h ^{-1} s_1 ^{-1} y)_o \geq C$$
is at most $\epsilon  \# S$. Hence, the number of $s = s_1 h s_2 \in S_1$ such that $(x, s ^{-1} y)_o \geq C$
is at most 
$$\epsilon (\# S)^2 \leq \frac{\epsilon}{(1-\epsilon)^2} \# S_1.$$
Thus, the set $S_1$ is $(\eta, C, D)$-Schottky if we set $\eta = \frac{2\epsilon}{(1-\epsilon)^2}$.

Now, let us check the coarse joint stabilizer. Let $s = s_1 h s _2$, and suppose $g$ belongs to $\textup{Stab}_K(o, s o)$.

Then, since $s_1 o$ and $s_1 h o$ lie within distance $C  + 4 \delta$ of any geodesic segment $\gamma = [o, so]$, then by Lemma \ref{L:fellowtravel} 
we have 
$$d(s_1 o, g s_1 o) \leq 2 C + 2 K + 12 \delta, \qquad d(s_1 h o, g s_1 h o) \leq 2 C + 2 K + 12 \delta$$
thus $g$ belongs to $\textup{Stab}_{K_1}(s_1 o, s_1 h o)$ with $K_1 = 2 C + 2 K  +  12 \delta$. 
Since $\textup{Stab}_{K_1}(s_1 o, s_1 h o)  = s_1 \textup{Stab}_{K_1}(o, h o) s_1^{-1}$ is finite by the WPD property, 
the claim follows. 

\end{proof}

From the previous lemma we immediately get: 

\begin{coro} \label{WPDSchottky2}
	Let $ \mu $ be a non-elementary measure on a countable set of isometries of a $ \delta $-hyperbolic space $ X $ such that the semigroup generated by the support of $ \mu $ contains a WPD element. For all $ \eta > 0 $, there exists $ C > 0 $ such that, for all $ D > 0 $ and $ K > 0 $, there exists $ N > 0 $ and an $(\eta, C, D)-$Schottky set $ S $ in the support of $ \mu ^{*N} $ with the property that for any $s \in S$ the coarse joint stabilizer 
$\textup{Stab}_K(o, s o)$ is finite.
\end{coro}

To prepare for the next proof, let us recall that in a hyperbolic metric space, any two points $o \in X$ and $\xi \in \partial X$ are connected by a $(1, 10 \delta)$-quasigeodesic ray (if the space is not proper, they may not be connected 
by a geodesic ray). 
Now, recall that from Lemma \ref{L:Chain} and the discussion after it, if we pick an $(\epsilon, C, D)$-Schottky set, then for any pivotal time from infinity $ k $, the points $ y _{k} ^{-}, y _{k} $ are within a $M$-neighbourhood of any $(1, 10 \delta)$-quasigeodesic $\gamma$ joining $o$ and the limit point $ \xi $, where $M = M(C, \delta)$ depends only on $ C $ and $ \delta $
(note that the constant $M$ may be slightly larger than the one in Eq. \eqref{E:fellow-t}). Also from Lemma \ref{L:Chain}, the points $ y _{k} ^{-},y _{k} $ are oriented in the ``right way" along $ \gamma$, 
namely any closest point projection of $y_{k}^{-}$ onto $\gamma$ lies between $o$ and any closest point projection of $y_k$ onto $\gamma$.

 \begin{proof}[Proof of Theorem \ref{T:WPD-action}]
Let $o \in X$ be a base point. From now on, we fix some $ (1/100, C, D)-$Schottky set $ S $ from Corollary \ref{WPDSchottky2} with $ D, K $ sufficiently large. For our purposes, it suffices to take $ D > 100C+100\delta+1 $ and $ K > 6 M + 1 + 30 \delta$ where $M = M(C, \delta)$ is as in the previous paragraph.

By \cite{Maher-Tiozzo}, the random walk $(w_n o)$ converges to $\partial X$ almost surely, and $(\partial X, \nu)$ is a $\mu$-boundary. Since by \cite{pivot}, pivots still exist and are abundant, the argument goes through in most parts, except for the following. First, we need to modify the definition of good distance to
  $$D_n^{\alpha,L} := \left\lfloor \sum_{\stackrel{I_{k, \alpha}, I_{k+1, \alpha} }{\textup{ $L$-good}}} d(p_k o, p_{k+1} o) \right\rfloor$$
 by adding the integer parts, and replacing the word metric by the distance $d$ on $X$.
The biggest modification in the argument is in Proposition \ref{P:pin-down}. 

We add one more partition as follows. At the final pivot, record the pivotal element $ s = (y _{k} ^{-}) ^{-1} y _{k} \in S $. If no pivots exist, record the identity element $ e $. Denote this partition by $ \mathcal{S} ^{\alpha}_{n} $. Since $ S$ is finite, this partition has entropy at most $ \log \left(\# S +1\right) $. \\

Now, let us condition our walk on $\mathcal{T}_n^\alpha \vee  D_n^{\alpha,L}  \vee \mathcal{B}_n^{\alpha,L} \vee \mathcal{S} ^{\alpha} _{n}$. 
We need to estimate the relative entropy
$$H_\xi(\mathcal{A}_n \vert  \mathcal{T}_n^\alpha \vee  D_n^{\alpha,L}  \vee \mathcal{B}_n^{\alpha,L}  \vee \mathcal{S} ^{\alpha} _{n}).$$

As in the proof of Proposition \ref{P:pin-down}, 	we know the distance between $ o $ and $ p _{last} o $ up to $ \frac{2 Mn}{\alpha} $, where $ M $ depends only on $ \delta $ and $ C $. 
We also know that $ p _{last} s o $ lies within distance $ M $ of any $(1, 10 \delta)$-quasigeodesic $\gamma$ connecting $o$ to $\xi$. 

Let us pick a parameterization $\gamma : [0, \infty) \to X$ of such a quasigeodesic. Since we know the distance between $ o $ and $ p _{last} o $ up to $ \frac{2 Mn}{\alpha}$, 
if we set $\Delta := \frac{2 Mn}{\alpha} +M + 1 +  10 \delta$, 
we know a parameter $t$ such that the interval $J := [t - \Delta, t + \Delta]$
is such that the $M$-neighbourhood $ N _{M} (\gamma(J)) $ contains all possible choices for $ p _{last} o $.

	Now cover $ J $ by $ \lceil  2 \Delta \rceil $ intervals of length $ 1 $, and for each interval $ J _{i}  $ find a group element $ g_i $ such that the point $ g_io $ is in $ N _{M} (\gamma(J _{i})) $, the point 
	$ g_i s o $ is contained in $ N _{M} ( \gamma(J)) $, and the projection of $ g_i  o $ onto $ \gamma $ lies between $o$ and the projection of $ g_i  s o $,
	if such a $ g _{i} $ exists. 
	
	Now for any choice $ p' $ for $ p _{last} $, we can find an interval $ J _{i} $ such that $ p'o \in N _{M} (\gamma(J _{i})) $. 
	Denote as $\pi$ a choice of a closest point projection onto $\gamma$, and let $\pi_i$ a choice of closest point projection onto $\gamma(J_i)$.
	Then 
	$$ d(p'o, g _{i}o) \leq d(\pi_i(p'o), \pi_i(g _{i}o)) + 2M \leq 1 + 10 \delta + 2M.$$ 
	We claim that $d(p'so, g_iso) \leq 6 M + 1 + 30 \delta$.
	
	To prove it, let $t_1, t_2, u_1, u_2$ be such that 
	$$\gamma(t_1) = \pi_i(p' o), \gamma(t_2) = \pi(p' s o), \gamma(u_1) = \pi_i(g_i o), \gamma(u_2) = \pi(g_i s o).$$	
	Note that $d(g_i o, g_i s o) = d(p' o, p' s o) = d(o, so)$ and, since the points $p' o, p' s o, g_i o$ and $g_i s o$ lie within distance $M$ of $\gamma$, 
	we have 
	$$|(t_2 - t_1) - d(o, so) | \leq |d(\pi_i(p' o), \pi(p' s o)) - d(o, so) | + 10 \delta \leq 2 M + 10 \delta$$
	hence similarly  
	$$ |(u_2 - u_1) - d(o, so) | \leq 2 M + 10 \delta$$
	since the projections of the segments $[p' o, p' so]$ and $[g_ i, g_i s o]$ have the same orientation, 
	that is $\max \{ t_1, u_1 \} < \min \{t_2, u_2\}$, 
	we also obtain 
		\begin{align*}
	|t_2 - u_2| & \leq |t_1 - u_1| + \left| |t_1 - t_2| - |u_1 - u_2| \right| \\
	& \leq 1 + 4 M + 20 \delta 
\end{align*}
so, by definition of quasigeodesic, 
$$d( \pi(p' s o), \pi(g_i s o)) \leq 1 + 4 M + 30 \delta.$$
Finally, we get
	\begin{align*}
d(p'so, g_iso)&  \leq  d( \pi(p' s o), \pi(g_i s o)) + 2 M \\
& \leq 1 + 6 M + 30 \delta.
\end{align*}
	Since $ K \geq 6 M  + 1 + 30 \delta$ then $\# \text{Stab} _{K} (o, so) = R < \infty$. Since we chose at most 
	$\lceil 2 \Delta \rceil \leq \frac{4 Mn}{\alpha} + 2 M + 3 +  20 \delta$ values of $ g _{i} $'s we have 
	$$H_\xi(\mathcal{A}_n \vert  \mathcal{T}_n^\alpha \vee  D_n^{\alpha,L}  \vee \mathcal{B}_n^{\alpha,L} \vee \mathcal{S} ^{\alpha} _{n}) \leq 
	\log( 4 R Mn / \alpha + (2 M + 3 +  20 \delta) R )$$
which, by dividing by $n$ and taking the limit as $n \to \infty$, implies the claim. 
\end{proof}

\bibliographystyle{amsalpha}
\bibliography{ref}

\end{document}